\documentclass[12pt]{amsart}
\usepackage{amsmath, amsthm, amsfonts, amssymb,mathrsfs}
\usepackage{enumerate}
\usepackage[bookmarks=false]{hyperref}
\usepackage{geometry}
\usepackage[utf8]{inputenc}
\usepackage[demo]{graphicx}
\usepackage[up]{caption}
\usepackage{subcaption}
\usepackage{pgf,tikz}
\usepackage[toc,page]{appendix}
\usetikzlibrary{arrows}
\usetikzlibrary{patterns}

\theoremstyle{plain}
\newtheorem{thm}{Theorem}[section]
\newtheorem*{thma}{Theorem A}
\newtheorem*{thmb}{Theorem B}
\newtheorem*{corc}{Corollary C}
\newtheorem*{thmd}{Theorem D}
\newtheorem{lem}[thm]{Lemma}
\newtheorem{prop}[thm]{Proposition}
\newtheorem{cor}[thm]{Corollary}

\theoremstyle{definition}
\newtheorem{defn}[thm]{Definition}
\newtheorem*{defn*}{Definition}

\newtheorem{exmp}[thm]{Example}
\newtheorem{rem}[thm]{Remark}

\numberwithin{equation}{section}

\newcommand{\N}{\mathbb{N}}

\newcommand{\Z}{\mathbb{Z}}
\newcommand{\T}{\mathbb{T}}
\newcommand{\cU}{\mathcal{U}}

\newcommand{\Stab}{\operatorname{Stab}}

\newcommand{\Sp}{\operatorname{Sp}}

\newcommand{\cay}{\operatorname{Cay}}

\newcommand{\diam}{\operatorname{diam}}

\newcommand{\rg}{\mathscr{G}}
\newcommand{\cf}{\mathcal{F}}

\begin{document}

\title[Maximal amenable subalgebras in group von Neumann algebras]{Maximal amenable subalgebras of von Neumann algebras associated with hyperbolic groups}

\begin{abstract}
We prove that for any infinite, maximal amenable subgroup $H$ in a hyperbolic group $G$, the von Neumann subalgebra $LH$ is maximal amenable inside $LG$. It provides many new, explicit examples of maximal amenable subalgebras in II$_1$ factors. We also prove similar maximal amenability results for direct products of relatively hyperbolic groups and orbit equivalence relations arising from measure-preserving actions of such groups.
\end{abstract}

\author{R\'emi Boutonnet and Alessandro Carderi}
\address{ENS Lyon \\
UMPA UMR 5669 \\
69364 Lyon cedex 7 \\
France}
\email{rboutonnet@ucsd.edu}
\email{alessandro.carderi@ens-lyon.fr}
\thanks{Research partially supported by ANR grant NEUMANN}

\maketitle

\section{Introduction}

Hyperfinite von Neumann algebras form the simplest and most fundamental class of von Neumann algebras.  This class is very well understood: Murray and von Neumann proved that there is a unique hyperfinite II$_1$ factor and Connes celebrated result \cite{Co76} states that hyperfinite von Neumann algebras are exactly the amenable ones. This characterization implies in particular that all von Neumann subalgebras of a hyperfinite tracial von Neumann algebra are completely described: they are hyperfinite. Up to now, such an understanding of subalgebras is out of reach for a non-hyperfinite von Neumann algebra.

Thus given a II$_1$ factor, it is natural to study the structure of its hyperfinite subalgebras. In the sixties, Kadison adressed a general question: is any self-adjoint element in a II$_1$ factor $M$ contained in a hyperfinite subfactor of $M$?
A first answer to this question was provided by Popa, who showed \cite{Po83} that the von Neumann subalgebra of $L\mathrm{F}_n$ ($n \geq 2$) generated by one of the generators of $\mathrm{F}_n$ is maximal amenable, and yet it is abelian.

Recently there has been some further work in this direction. In 2006, Shen \cite{Sh06} extended the work of Popa to countable direct products of free group factors, providing the first example of an abelian maximal amenable subalgebra in a McDuff factor. Subsequently, the authors in \cite{CFRW08} investigated the radial subalgebra of $L\mathrm{F}_n$ and they managed to prove that it is also maximal amenable.
In 2010, Jolissaint \cite{Jo10} extended Popa's result, providing examples of maximal amenable subalgebras in factors associated to amalgamated free-product groups, over finite subgroups.
Infinitely many explicit examples of maximal amenable subalgebras were also discovered by Houdayer \cite{Ho14a}.  He showed that the measure class on $\T^2$ associated to a maximal amenable, abelian subalgebra in a II$_1$ factor reaches a wide range. An example from subfactor theory was also provided in \cite{Bro14}.

In this article, we prove a general rigidity result on maximal amenability from a ``group-subgroup'' situation, to the ``von Neumann algebra-subalgebra'' they generate. The main focus is on hyperbolic groups. At the group level, infinite amenable subgroups of hyperbolic groups are completely understood: they are virtually cyclic, and they act in a nice way on the Gromov boundary of the group. Using this fact, we can show the following, generalizing the main result of \cite{Po83}.

\begin{thma}
Consider a hyperbolic group $G$ and an infinite, maximal amenable subgroup $H < G$. 
Then the group von Neumann algebra $LH$ is maximal amenable inside $LG$.
\end{thma}

This answers a question of Cyril Houdayer \cite[Problème 3.13]{Ho13}.

Since any maximal amenable subgroup $H$ of a hyperbolic group is virtually cyclic, the associated von Neumann algebra $LH$ is far from being a factor. By Remark \ref{unique}, we obtain many counterexamples to Kadison's question, even in property (T) factors. For instance factors of the form $L\Gamma$, with $\Gamma$ a cocompact lattice in $\Sp(n,1)$, are counterexamples with property (T).

The proof of Theorem A is in the spirit of Popa's asymptotic orthogonality property \cite{Po83}. It relies on an analysis of $LH$-central sequences and property Gamma. By definition, a diffuse finite von Neumann algebra $M$ has \textit{property Gamma} if it admits a sequence of unitaries $(u_n)_n\subset M$ which tends weakly to $0$ such that for every $x\in M$, \[\lim_n\|xu_n-u_nx\|_2= 0.\]

By \cite{Co76}, diffuse finite amenable von Neumann algebras have property Gamma. What we really show is that $LH\subset LG$ is maximal Gamma, that is, it is maximal among von Neumann subalgebras of $LG$ with property Gamma. However amenability and property Gamma coincide for subalgebras of solid von Neumann algebras (\cite[Proposition 7]{Oz04}) and the main result of \cite{Oz04} shows that $LG$ is solid, whenever $G$ is hyperbolic.

Using similar techniques, we can prove the following result for relatively hyperbolic groups. 

\begin{thmb}
Let $G$ be a group which is hyperbolic relative to a family $\rg$ of subgroups of $G$ and consider an infinite subgroup $H \in \rg$ such that $LH$ has property Gamma. 
Then the group von Neumann algebra $LH$ is maximal Gamma inside $LG$.
\end{thmb}

Using results of Osin \cite{Os06a,Os06b}, we obtain the following corollary, which generalizes Theorem A and the main result of \cite{Jo10}. 

\begin{corc}
Let $G$ be a group which is hyperbolic relative to a family $\rg$ of amenable subgroups and $H$ be an infinite maximal amenable subgroup of $G$. 
Then the group von Neumann algebra $LH$ is maximal amenable inside $LG$.
\end{corc}

By the comments after Proposition 12 in \cite{Oz06}, $LG$ is solid for $G$ as in Corollary C, so maximal amenable is equivalent to maximal Gamma. 

Limit groups are examples of groups $G$ covered by this corollary.

It is also possible to prove similar results in the context of hyperbolically embedded subgroups, in the sense of \cite{DGO11}: generalizing our techniques one can show that if $H < G$ is an infinite amenable subgroup which is hyperbolically embedded then $LH$ is maximal amenable inside $LG$.

Finally, we extend our results to products of groups as above. We also allow the groups to act on an amenable von Neumann algebra, and we get a similar result about the crossed product von Neumann algebra. Such a product situation were already investigated in \cite{Sh06} and \cite{CFRW08}. We thank Stuart White for suggesting us to study this case. 

\begin{thmd}
Let $n \geq 1$, and consider for all $i = 1,\dots,n$ an inclusion of groups $H_i < G_i$ as in Theorem B. Put $G := G_1 \times \cdots \times G_n$ and $H := H_1 \times \cdots \times H_n$.

Then for any trace-preserving action of $G$ on a finite amenable von Neumann algebra $(Q,\tau)$, the crossed-product $Q \rtimes H$ is maximal amenable inside $Q \rtimes G$.
\end{thmd}

In particular, when $G$ and $H$ are as above, for any free measure preserving action on a probability space $G \curvearrowright (X,\mu)$, the equivalence relation on $(X,\mu)$ given by the $H$-orbits is maximal hyperfinite inside the equivalence relation given by the $G$-orbits.

In Theorem D, note that $Q \rtimes H \subset Q \rtimes G$ is not maximal Gamma in general. We will in fact use Houdayer's relative version of the asymptotic orthogonality property to conclude (\cite{Ho14b}). The argument relies on the same analysis of $LH$-central sequences.

\subsection*{Organization of the article}

Apart from the introduction, the paper contains four other sections. Section 2 gives a criterion to prove maximal amenability results in the context of group von Neumann algebras. It also contains the main definitions and tools regarding relatively hyperbolic groups. In section 3, we provide a proof of Theorem A for the reader which is not familiar with relatively hyperbolic groups. Section 4 turns to the general case of relatively hyperbolic groups and the proofs of Theorem B and Corollary C. Finally Section 5 is devoted to the proof of Theorem D.

\subsection*{Acknowledgement}

We are very grateful to Cyril Houdayer for showing \cite[Problème 3.13]{Ho13} to us and for his valuable comments helping us to improve the presentation of this work. We also warmly thank Stuart White for useful remarks and for attracting our attention on Theorem D. We finally thank the referees for useful comments.


\section{Preliminaries}

\subsection{Central sequences and group von Neumann algebras}
\label{central sequences}

In this section, we consider an inclusion of two countable discrete groups $H < G$. We denote by $LH \subset LG$ the associated von Neumann algebras and by $u_g$ the canonical unitaries in $LG$ that correspond to elements $g \in G$.

For a set $F \subset G$, we will by denote $P_F : \ell^2(G) \to \ell^2(F)$ the orthogonal projection onto $\ell^2(F)$.

As explained in the introduction, the proofs of our main results rely on an analysis of $LH$-central sequences. We describe here how the $H$-conjugacy action on $G$ allows localizing the Fourier coefficients of $LH$-central sequences in terms of projections $P_F$, $F \subset G$.

\begin{defn}
\label{Hroaming}
Let $H < G$ be an inclusion of two countable groups.
A set $F \subset G \setminus H$ is said to be {\it $H$-roaming} if there exists an infinite sequence $(h_k)_{k \geq 0}$ of elements in $H$ such that \[h_kFh_k^{-1} \cap h_{k'}Fh_{k'}^{-1} = \emptyset \text{ for all } k \neq k'.\]
Such a sequence $(h_k)_k$ is called a {\it disjoining sequence}.
\end{defn}

The following standard lemma is the key of our proofs.

\begin{lem}
\label{lemmalnormal}
Let $H < G$ be an inclusion of two countable groups and denote by $LH \subset LG$ the associated von Neumann algebras. Assume that $(x_n)_n$ is a bounded $LH$-central sequence in $LG$.

Then for any $H$-roaming set $F$ we have that $\lim_n \Vert P_F(x_n)\Vert_2 = 0$.
\end{lem}
\begin{proof}
Assume that $F$ is an $H$-roaming set and consider a disjoining sequence $(h_k)_k \subset H$ for $F$. Since $(x_n)_n$ is $LH$-central, we have for all $k$
\begin{equation}\label{limsup1}
\limsup_n \Vert P_F(x_n) \Vert_2  = \limsup_n \Vert P_F(u_{h_k}x_n u_{h_k}^*)\Vert_2
 = \limsup_n \Vert P_{h_k^{-1}Fh_k}(x_n) \Vert_2.
\end{equation}
But $P_{h_k^{-1}Fh_k}(x_n) \perp P_{h_{k'}^{-1}Fh_{k'}}(x_n)$ for all $k \neq k'$ and all $n$. 
Thus we get that for any $N \geq 0$ and $n \geq 0$, 
\[ \Vert x_n \Vert_\infty^2 \geq \Vert x_n \Vert_2^2 \geq \sum_{k \leq N} \Vert P_{h_k^{-1}Fh_k}(x_n) \Vert_2^2.\]
Applying \eqref{limsup1}, we deduce that $\sup_n \Vert x_n \Vert_\infty^2 \geq N\limsup_n \Vert P_F(x_n) \Vert_2^2 $. 
Since $N$ can be arbitrarily large, we get the result.
\end{proof}

\begin{prop}
\label{prop:condition}
Let $H < G$ be an inclusion of two infinite countable groups. Assume that for any $s,t \in G \setminus H$, there exists an $H$-roaming set $F \subset G\setminus H$ such that $sF^ct\cap F^c$ is finite.

If $LH$ has property Gamma, then it is maximal Gamma inside $LG$.
\end{prop}
\begin{proof}
Assume that there exists an intermediate von Neumann algebra $P$ with property Gamma: $LH \subset P \subset LG$ and consider a central sequence $(v_n)_n\subset P$ of unitary elements which tends weakly to $0$. 

We have to show that any $a \in P \ominus LH$ is equal to $0$. On the one hand, we have that $\lim_n \langle av_na^*,v_n \rangle = \Vert a \Vert_2^2$ because the unitaries $v_n$ asymptotically commute with $a$. On the other hand we claim that this limit is $0$. 

Indeed, by a standard linearity/density argument, it is sufficient to check that for all $s, t \notin H$, we have $\lim_n \langle u_sv_nu_t,v_n \rangle = 0$.

So fix $s,t \in G \setminus H$. By assumption there exists an $H$-roaming set $F$ such that $K := sF^ct\cap F^c$ is finite. Since $(v_n)_n$ is $LH$-central and bounded, Lemma \ref{lemmalnormal} implies that $\lim_n \Vert P_F(v_n)\Vert_2 = 0$.
Noting that $u_sP_{F^c}(v_n)u_t$ is in the range of $P_{sF^ct}$ for all $n$, we obtain
\begin{align*}
\limsup_n \vert \langle u_sv_nu_t,v_n \rangle \vert & = \limsup_n \vert \langle u_sP_{F^c}(v_n)u_t,P_{F^c}(v_n)\rangle \vert\\
& = \limsup_n \vert \langle u_sP_{F^c}(v_n)u_t,P_{sF^ct}\circ P_{F^c}(v_n)\rangle \vert\\
& \leq \limsup_n \Vert P_K(v_n) \Vert_2=0,
\end{align*}
because $(v_n)_n$ tends weakly to $0$ and $K$ is finite.
\end{proof}

If $H < G$ is an inclusion satisfying the assumption of Proposition \ref{prop:condition}, then $H$ is {\it almost malnormal} in $G$ in the sense that $sHs^{-1} \cap H$ is finite for all $s \notin H$ (or equivalently $sHt \cap H$ is finite for all $s,t \notin H$). In terms of von Neumann algebras this translates as follows.

\begin{prop}
\label{propmixing}
A subgroup $H$ of a group $G$ is almost malnormal if and only if the von Neumann subalgebra $LH \subset LG$ is mixing, meaning that $\lim_n \Vert E_{LH}(av_nb)\Vert_2 = 0$, for all $a,b \in LG \ominus LH$ and for any sequence $(v_n)_n \subset \cU(LH)$ which tends weakly to $0$. 
\end{prop}
\begin{proof}
Assume that $H$ is not almost malnormal inside $G$: there exists $s \in G \setminus H$ such that $sHs^{-1} \cap H$ contains a sequence going to infinity $(h_n)_n$. For all $n$ put $v_n := u_{h_n} \in \cU(LH)$ and put $a^* = b = u_s \in LG \ominus LH$. Then $LH \subset LG$ is not mixing, because the sequence $(v_n)$ goes weakly to $0$, whereas 
\[\limsup_n \Vert E_{LH}(av_nb) \Vert_2 = \limsup_n \Vert E_{LH}(u_{s^{-1}h_ns}) \Vert_2 = \limsup_n \Vert u_{s^{-1}h_ns} \Vert_2 = 1 \neq 0.\]

Conversely assume that $H$ is almost malnormal inside $G$. Take a sequence $(v_n)_n \subset \cU(LH)$ which tends weakly to $0$. If $s,t \in G \setminus H$, then $K := s^{-1}Ht^{-1} \cap H$ is finite. Hence 
\[\limsup_n \Vert E_{LH}(u_sv_nu_t)\Vert_2 = \limsup_n \Vert P_{s^{-1}Ht^{-1}}(v_n) \Vert_2 = \limsup_n \Vert P_K(v_n)\Vert_2 = 0.\]
This implies the mixing property, because $\{u_s \, , \, s \in G \setminus H\}$ spans a $\Vert \cdot \Vert_2$-dense subset of $LG \ominus LH$.
\end{proof}

\subsection{Relatively hyperbolic groups and their boundary}
\label{section:relhyp}

The contents of this section is taken from Bowditch \cite{Bo12}. Let us fix first some terminology and notations about graphs. 

Let $K$ be a connected graph. Its vertex set and edge set are denoted by $V(K)$ and $E(K)$ respectively. A {\it path} of length $n$ between two vertices $x$ and $y$ is a sequence $(x_0,x_1,\dots,x_n)$ of vertices such that $x_0 = x$ and $x_n = y$, and $(x_i,x_{i+1}) \in E(K)$ for all $i = 0,\dots, n-1$. The path $(x_0,\dots,x_n)$ is a {\it loop} if $x_0 = x_n$ and if $x_0,x_1, \dots, x_n$ are distinct. For a path $\alpha = (x_0,x_1,\dots,x_n)$, we put $\alpha(k) = x_k$, $k = 0,\dots,n$.

We endow $K$ with the distance $d$ given by the length of a shortest path between two points.
A path $\alpha$ between two vertices $x$ and $y$ is a {\it geodesic} if its length equals to $d(x,y)$. We denote by $\cf(x,y)$ the set of all geodesics between $x$ and $y$.

More generally, for $r \geq 0$, a  path $\alpha$ is an {\it $r$-quasi-geodesic} if all its vertices are distinct, and if for any finite subpath $\beta = (x_0,\dots,x_n)$ of $\alpha$, the length of $\beta$ is smaller than $d(x_0,x_n) + r$. 
Note that the geodesics are exactly the $0$-quasi-geodesics. For $x,y \in V(K)$, denote by $\cf_r(x,y)$ the set of $r$-quasi-geodesics between $x$ and $y$.

We will also consider infinite paths $(x_0,x_1,\dots)$ or bi-infinite paths $(\dots,x_{-1},x_0,x_1,\dots)$. For $r \geq 0$, such an infinite or bi-infinite path will be called $r$-quasi-geodesic if all its finite subpaths are $r$-quasi-geodesics.

In a graph $K$,  a {\it geodesic triangle} is a set of three vertices $x,y,z \in V(K)$, together with geodesic paths $[x,y] \in \cf(x,y)$, $[y,z] \in \cf(y,z)$ and $[z,x] \in \cf(z,x)$ connecting them. These paths are called the {\it sides} of the triangle.

\begin{defn}[Gromov \cite{Gr87}]
A connected graph $K$ is called {\it hyperbolic} if there exists a constant $\delta > 0$ such that every geodesic triangle in $K$ is {\it $\delta$-thin}: each side of the triangle is contained in the $\delta$-neighbourhood of the union of the other two, namely $[x,y] \subset B([y,z]\cup [z,x],\delta)$, and similarly for the other two sides.
\end{defn}

Two infinite quasi-geodesics in a hyperbolic graph $K$ are {\it equivalent} if their Hausdorff distance is finite. The {\it Gromov boundary} $\partial K$ of $K$ is the set of equivalence classes of infinite quasi-geodesics. The endpoints of a path $\alpha = (x_0,x_1,\dots)$ in a class $x$ of $\partial K$ are defined to be $x_0$ and $x$. Similarly, a bi-infinite path $\alpha = (\dots,x_{-1},x_0,x_1,\dots)$ has endpoints $\alpha_- := [(x_0,x_{-1},\dots)] \in \partial K$ and $\alpha_+ := [(x_0,x_1,\dots)] \in \partial K$. It turns out that for any two points $x,y \in K \cup \partial K$, for any $r\geq 0$, the set $\cf_r(x,y)$ of $r$-quasi-geodesics connecting them is non-empty.

Recall that a {\it hyperbolic group} is a finitely generated group
which admits a hyperbolic Cayley graph (this implies that all its
Cayley graphs are hyperbolic). We will define similarly relatively
hyperbolic groups, but we have to replace the Cayley graph by a graph
in which some subgroups are ``collapsed" to points. 

\begin{defn}[\cite{Fa98}]
Consider a group $G$, with finite generating set $S$ and denote by $\Gamma := \cay(G,S)$ the associated Cayley graph. Let $\rg$ be a collection of subgroups of $G$. The {\it coned-off graph of $\Gamma$ with respect to $\rg$} is the graph $\hat{\Gamma}$ with:
\begin{itemize}
\item vertex set $V(\hat{\Gamma}) := V(\Gamma) \sqcup \bigsqcup_{H \in \rg} G/H$;
\item edge set $E(\hat{\Gamma}) := E(\Gamma) \sqcup \{(gh,[gH])\, \vert \, H \in \rg, [gH] \in G/H, h \in H\}$.
\end{itemize}
\end{defn}

In the sequel, we will identify $V(\Gamma)$ with $G$. The action of $G$ on itself by left multiplication extends to an isometric action on $\hat\Gamma$. The stabilizer of the vertex $[gH]$ is equal to $gHg^{-1}$.

Note that in the coned-off graph $\hat{\Gamma}$, the distance between two elements $g$ and $gh$ is at most $2$ whenever $h \in H$ for some $H \in \rg$. Note also that this coned-off graph will not be locally finite in general. But it will sometimes satisfy the following {\it fineness} condition. 

\begin{defn}[\cite{Bo12}]
\label{fine}
A graph $\Gamma$ is called {\it fine} if each edge of $\Gamma$ is contained in only finitely many loops of length $n$, for any given integer $n$.
\end{defn}

\begin{defn}[\cite{Bo12}]
  A group $G$ is said to be {\it hyperbolic relative to the
  family $\rg$} if there exists a finite generating set $S$ of $G$ such that
  the coned-off graph $\hat\Gamma$ is fine and $\delta$-hyperbolic
  (for some $\delta \geq 0$). 
\end{defn}

From this definition, usual hyperbolic groups appear as hyperbolic relative to the empty family.
For any relatively hyperbolic group $G$ with Cayley graph $\Gamma$, let us define a topology on $\Delta \Gamma :=\hat\Gamma\cup\partial\hat\Gamma$. 

\begin{defn}
  Given $x\in \Delta\Gamma$ and a finite set $A \subset
  V(\hat\Gamma)$ such that $x \notin A$, we define \[M(x,A):=\left\{y\in\Delta\Gamma : A \cap
    \alpha = \emptyset, \forall \alpha\in \cf(x,y)\right\}.\]
\end{defn}

\begin{thm}[\cite{Bo12}, section $8$]\label{topology}
The family $\{M(x,A)\}_{x,A}$ is a basis for a Hausdorff compact
topology on $\Delta\Gamma$ such that $G\subset \Delta\Gamma$ is
a dense subset, and every graph automorphism of $\hat\Gamma$ extends to a homeomorphism of $\Delta\Gamma$.
\end{thm}

Actually, we will not use the fact that $\Delta \Gamma$ is compact. The proof of Theorem \ref{topology} relies on the following lemma, which will be our main tool in order to manipulate neighbourhoods in $\Delta \Gamma$.

\begin{lem}[\cite{Bo12}, Section 8]\label{lem:key}
Let $r \geq 0$. The following facts are true.
\begin{enumerate}
\item For every $x,y\in \Delta\Gamma$, the graph $\bigcup_{\alpha \in \cf_r(x,y)} \alpha$ is locally finite.
\item For every edge $e\in E(\hat\Gamma)$, there exists a finite set
$E_r(e)\subset E(\hat\Gamma)$ such that for all $x,y \in
\Delta\Gamma$, and all $\alpha,\beta\in \cf_r(x,y)$ with $e \in
\alpha$, we have that $E_r(e) \cap \beta$ contains at least one edge.
\item For every $a\in V(\hat\Gamma), x \in \Delta \Gamma$, with $x \neq a$, there exists a finite set $V_{r,x}(a)\subset V(\hat\Gamma) \setminus \{ x \}$ such that for all $y \in \Delta \Gamma$, and all $\alpha,\beta\in \cf_r(x,y)$ with $a \in \alpha$, we have that $\beta \cap V_{r,x}(a) \neq \emptyset$.
\end{enumerate} 
\end{lem}
\begin{proof} 
  The first two facts are Lemma 8.2 and Lemma 8.3 in \cite{Bo12}. To derive the
  third fact from the others, fix $a \in V(\hat\Gamma)$ and $x \in
  \Delta\Gamma$. Denote by $E_0$ the set of edges $e$ in the graph $\bigcup_{\alpha \in \cf_r(a,x)} \alpha$ such that
  $a$ is an endpoint of $e$. By (1), the set $E_0$ is
  finite. Now put $E := \bigcup_{e \in E_0} E_r(e)$, and define
  $V_{r,x}(a)$ to be the set of endpoints of $E$, from which we remove $x$ if necessary. This is a finite set.

  Now if $\alpha \in \cf_r(x,y)$ goes through $a$, then it will contain
  an edge in $E_0$. Thus any $\beta \in \cf_r(x,y)$ contains an edge in
  $E$, and we are done by the definition of $V_{r,x}(a)$.
\end{proof}

Lemma \ref{lem:key} will always be used via the following easy corollary.

\begin{cor}\label{cor:bigA}
Let $r > 0$, $x \in \Delta \Gamma$ and $A \subset V(\hat{\Gamma}) \setminus \{x \}$ finite. Then there exists a set $V_{r,x}(A)$ containing $A$ such that,
\begin{enumerate}
\item if $y \notin M(x,A)$, then any $r$-quasi-geodesic $\alpha \in \cf_r(x,y)$ intersects $V_{r,x}(A)$,
\item if $y \in M(x,V_{r,x}(A))$, then no $r$-quasi-geodesic from $y$ to $x$ intersect $A$. 
\end{enumerate}
\end{cor}
\begin{proof}
  The set $V_{r,x}(A):=A\cup \bigcup_{a\in A} V_{r,x}(a)$ does the job. 
\end{proof}

Now we describe a way of constructing quasi-geodesic paths. The following lemma is well known but we include a proof for convenience. 

\begin{lem}\label{concatenation}
  There exist a constant $r_0 \geq 0$, only depending on the hyperbolicity constant of the graph $\hat \Gamma$, with the following property: for any geodesic paths $\alpha, \beta$ sharing exactly one endpoint $a$, if $a$ is the closest point of $\alpha$ to each point of $\beta$, then $\alpha \cup \beta$ is an $r_0$-quasi-geodesic.
\end{lem}
\begin{proof}
  We can take $r_0=8\delta+8$. Indeed, given $x\in \alpha$ and $y\in\beta$, we denote with $[x,y]\in\cf(x,y)$ a geodesic between $x$ and $y$. Since the triangle $\{x,y,a\}$ is $\delta$-thin, there exists $z\in [x,y]$ and $u\in\alpha$, $v\in\beta$ such that $d(z,u)\leq\delta+1$ and $d(z,v)\leq\delta+1$. By hypothesis on $a$, we have that $d(u,a)\leq 2\delta+2$ and hence $d(v,a)\leq 4\delta +4$. Finally, 
  \begin{align*}
    d(x,a)+d(y,b)\leq& (d(x,z)+d(z,u)+d(u,a))+(d(a,v)+d(v,z)+d(z,y))\\
    \leq& d(x,z)+d(y,z)+8\delta+8=d(x,y)+8\delta+8.\qedhere
  \end{align*}
\end{proof}

\begin{defn} \label{perpendicular}
  Consider $x,y,z \in \Delta \Gamma$ and let $\alpha \in \cf(x,y)$ be a geodesic. A point $z_0 \in \alpha$ which minimizes the distance from $z$ to $\alpha$, is called a {\it projection of $z$ on $\alpha$}. Such a $z_0$ splits the path $\alpha$ into two geodesic paths $\alpha_x \in \cf(x,z_0)$ and $\alpha_y \in \cf(z_0,y)$. Given any geodesic $\beta\in\cf(z,z_0)$, we can join $\beta$ and $\alpha_x$ or $\alpha_y$ to get two paths that are $r_0$-quasi geodesic by Lemma \ref{concatenation}. 
\end{defn}

We end this section with a lemma that we will need later and which is essentially contained in Section 8 of \cite{Bo12}. Its proof illustrates well how to use the tools introduced above.

\begin{lem}
\label{uniformbasis}
  For every $x \in\Delta\Gamma$ and for every finite subset $A\subset V(\hat\Gamma)\setminus \{x \}$, there exists a finite subset $C \subset V(\hat\Gamma)\setminus \{ x \}$ such that for every $y\in M(x,C)$, \[M(y,C)\subset M(x,A).\] 
\end{lem}
\begin{proof}
  Let $r_0 \geq 0$ be given by Lemma \ref{concatenation}, and set $C:=V_{r_0,x}(V_{r_0,x}(A))$ (see Corollary \ref{cor:bigA}). We will show that the conclusion of the lemma holds for this $C$.

If $y = x$, we see  that $M(x,C) \subset M(x,A)$ because $A \subset V_{r_0,x}(A) \subset C$.
Now let $y \in M(x,C)$, with $y \neq x$, and take $z \notin M(x,A)$.
We will show that $z \notin M(y,C)$.

Let $\alpha$ be a geodesic between $y$ and $z$. Consider a projection $x_0$ of $x$ on $\alpha$ as in Definition \ref{perpendicular} and let $\beta \in \cf(x,x_0)$. We denote with $\alpha_y$ (resp. $\alpha_z$) the subgeodesic of $\alpha$ between $x_0$ and $y$ (resp. $x_0$ and $z$). 
Then, by Lemma \ref{concatenation} the paths $\beta \cup \alpha_y \in \cf_{r_0}(x,y)$ and $\beta \cup \alpha_z \in \cf_{r_0}(x,z)$ are $r_0$-quasi-geodesics.

Since $z\notin M(x,A)$, Corollary \ref{cor:bigA}(1) implies that $\beta \cup \alpha_z$ intersects $V_{r_0,x}(A)$. If the intersection point lied on $\beta \cup \alpha_y$, then Corollary \ref{cor:bigA}(2) would contradict our assumption that $y \in M(x,C)$. Hence the intersection point lies on $\alpha_z \subset \alpha$. We have found a geodesic between $z$ and $y$ which intersects a point of $V_{r_0,x}(A) \subset C$, which means precisely that $z \notin M(y,C)$.
\end{proof}


\section{Hyperbolic case: proof of Theorem A}
\label{proofA}

Suppose that $G$ is a hyperbolic group and that $H$ is an infinite
maximal amenable subgroup of $G$. We want to apply Proposition \ref{prop:condition} in order to prove Theorem A.

As mentioned in Section \ref{section:relhyp}, $G$ is hyperbolic relative to the empty family and $\hat\Gamma = \Gamma$, for any Cayley graph $\Gamma$ of $G$. Thus $\Delta \Gamma := \Gamma \cup \partial \Gamma$ is the usual Gromov compactification of $\Gamma$, with boundary $\partial \Gamma$, endowed with the topology generated by the sets $\{M(x,A)\}_{x,A}$. As before, we identify $G$ with $V(\Gamma)$.

Recall that the action of $G$ by left multiplication on itself extends to a continuous action on $\Delta\Gamma$.

The amenable subgroup $H$ has a particular form by \cite[Théorème
8.29, Théorème 8.37]{GdH90}. First, $H$ admits an element $a \in H$
of infinite order which generates a finite index subgroup of $H$. Second, the element
$a$ acts on $\Delta \Gamma$ with exactly two fixed points $a_\pm
\in \partial\Gamma$, and $H \subset \Stab_G(\{ a_-,a_+ \})$. By maximal amenability of $H$ and since $\Stab_G(\{ a_-,a_+ \})$ is virtually cyclic, we have the equality $H = \Stab_G(\{ a_-,a_+\})$. Moreover the stabilizers of $a_-$ and $a_+$ are equal \cite[Théorème 8.30]{GdH90}, although they are not necessarily equal to $H$.

Also, $\Stab_G(a_-)$ and $\Stab_G(a_+)$ are contained in $H$ (but not necessarily equal). 

Moreover, the fixed points $a_\pm$ of $a$ are such that $\lim_{n \to +\infty} a^n x = a_+$ and $\lim_{n \to -\infty} a^n x = a_-$, for any $x \in \Delta \Gamma$ (so in particular $a_+$
is the unique cluster point of the sequence $\{a^n\}_{n\geq 0}$).

The action of $G$ on itself by right multiplication also extends to a continuous action on $\Delta\Gamma$, in such a way that any element $g \in G$ acts trivially on $\partial \Gamma$ (see for instance \cite[Proposition 5.3.18]{BO08}).

In order to find an $H$-roaming set as in Proposition \ref{prop:condition}, we need to understand geometrically the conjugacy action of $H$ on $G$. 
We start by collecting properties of left and right actions of $H$ on
$\Delta \Gamma$ separately, in the following two lemmas. Combining
these lemmas, we will see that the conjugacy action of $H$ has a
uniform {\it ‘‘north-south dynamics'' out of $H$}, as shown in Figure \ref{fighyp:left}.

\begin{figure*}[h]
\centering
\begin{subfigure}{.5\textwidth}
  \centering
  \begin{tikzpicture}[line cap=round,line join=round,>=triangle 45,x=1.0cm,y=1.0cm]
\clip(-3.8,-3.5) rectangle (5,3.53);
\draw(0,0) circle (3cm);
\draw[black,domain=-2.683:2.683,rotate around={90:(0,0)}] plot({\x},{sqrt(7.5+0.5*(\x)^2)-2});
\draw[black,domain=-2.683:2.683,rotate around={270:(0,0)}] plot({\x},{sqrt(7.5+0.5*(\x)^2)-2});

\draw (-3,0)-- (3,0);
\draw[black,domain=-1.02:0.1,rotate around={90:(0,0)}] plot({\x},{sqrt(10+(\x)^2)-0.5});
\draw[black,domain=-2.461:0.1, rotate around={90:(0,0)}]
plot({\x},{-sqrt(40+(\x)^2)+8.5});

\draw[black,domain=0.43:1.02,rotate around={90:(0,0)}] plot({\x},{sqrt(10+(\x)^2)-0.5});
\draw[black,domain=0.43:2.461, rotate around={90:(0,0)}] plot({\x},{-sqrt(40+(\x)^2)+8.5});

\draw[->] (0,0) ++ (30:3.5) arc (30:165:3.5);
\draw[->] (0,0) ++ (-30:3.5) arc (-30:-165:3.5);

\fill[pattern = north east lines, opacity=0.2]
plot[domain=-2.683:2.683,rotate around={270:(0,0)}]
({\x},{sqrt(7.5+0.5*(\x)^2)-2}) -- plot[domain=270+26.39:270-26.39]
({3*cos(\x)},{3*sin(\x)}) -- plot[domain=-2.683:2.683,rotate around={90:(0,0)}]
({\x},{sqrt(7.5+0.5*(\x)^2)-2}) -- plot[domain=90+26.39:90-26.39]
({3*cos(\x)},{3*sin(\x)});

\fill[pattern = north east lines, opacity=0.2]
plot[domain=-1.02:1.02,rotate around={90:(0,0)}]
({\x},{sqrt(10+(\x)^2)-0.5}) -- plot[domain=180-55.19:180-55.19+35.35]
({3*cos(\x)},{3*sin(\x)}) -- plot[domain=2.461:-2.461, rotate
around={90:(0,0)}] ({\x},{-sqrt(40+(\x)^2)+8.5}) -- plot[domain=180+55.19:180+55.19-35.35]
({3*cos(\x)},{3*sin(\x)});

\fill [color=black] (-3,0) circle (1.5pt);
\draw[color=black] (-3.3,-0.1) node {$a_+$};
\fill [color=black] (3,0) circle (1.5pt);
\draw[color=black] (3.4,-0.1) node {$a_-$};

\draw (1.8,-0.3) node {$H$};
\draw (-1.5,-0.3) node {$H$};
\draw (0,1.5) node {$F$};
\draw (-1.4,1.5) node {$V$};
\draw (1.8,1.5) node {$V$};
\draw (-2.25,0.3) node {$hFh^{-1}$};
\draw (2.7,-3) node {$h\cdot h^{-1}$};
\draw (2.7,3) node {$h\cdot h^{-1}$};
\end{tikzpicture}
   \caption{Conjugacy action of $H$ on $G$.}
  \label{fighyp:left}
\end{subfigure}%
\begin{subfigure}{.5\textwidth}
  \centering
  \begin{tikzpicture}[line cap=round,line join=round,>=triangle 45,x=1.0cm,y=1.0cm]
\clip(-4,-3.73) rectangle (5.96,3.32);
\draw(0,0) circle (3cm);
\draw[black,domain=-1.38:1.38,rotate around={90:(0,0)}]
  plot({\x},{sqrt(7.5+(\x)^2)-0.4});
\draw[black,domain=-1.38:1.38,rotate around={270:(0,0)}] plot({\x},{sqrt(7.5+(\x)^2)-0.4});
\draw (-3,0)-- (3,0);

\draw (0,-3.46) ++(30:1.73) arc (30:150:1.73);

\draw[black,domain=-0.91:0,rotate around={150:(0,0)}]
  plot({\x},{sqrt(7.5+3*(\x)^2)-0.3});
\draw[black,domain=-0.91:-0.75,rotate around={210:(0,0)}]
plot({\x},{sqrt(7.5+3*(\x)^2)-0.3});

\draw[black,domain=0.7:0.91,rotate around={150:(0,0)}]
  plot({\x},{sqrt(7.5+3*(\x)^2)-0.3});
\draw[black,domain=-0.1:0.91,rotate around={210:(0,0)}]
plot({\x},{sqrt(7.5+3*(\x)^2)-0.3});

\fill[pattern = north east lines, opacity=0.2] plot[domain=-0.91:0.91,rotate around={150:(0,0)}]
  ({\x},{sqrt(7.5+3*(\x)^2)-0.3})-- plot[domain=240-17.64:240+17.64]
({3*cos(\x)},{3*sin(\x)});

\fill[pattern = north east lines, opacity=0.2] plot[domain=-0.91:0.91,rotate around={210:(0,0)}]
  ({\x},{sqrt(7.5+3*(\x)^2)-0.3})-- plot[domain=300-17.64:300+17.64]
({3*cos(\x)},{3*sin(\x)});

\fill[pattern = north east lines, opacity=0.2] plot[domain=-1.38:1.38,rotate around={90:(0,0)}] ({\x},{sqrt(7.5+(\x)^2)-0.4})--plot[domain=180+27.32:180-27.32]
({3*cos(\x)},{3*sin(\x)});

\fill[pattern = north east lines, opacity=0.2] plot[domain=-1.38:1.38,rotate around={270:(0,0)}] ({\x},{sqrt(7.5+(\x)^2)-0.4})--plot[domain=27.32:-27.32]
({3*cos(\x)},{3*sin(\x)});

\fill [color=black] (-1.5,-2.6) circle (1.5pt);
\fill [color=black] (1.5,-2.6) circle (1.5pt);
\fill [color=black] (-3,0) circle (1.5pt);
\fill [color=black] (0,0) circle (1.5pt);
\fill [color=black] (0,-3.46+1.73) circle (1.5pt);
\draw[->] (0,-0.1) -- (0,-3.35+1.73);
\fill [color=black] (3,0) circle (1.5pt);

\draw[color=black] (-3.3,-0.1) node {$a_+$};
\draw[color=black] (3.4,-0.1) node {$a_-$};
\draw[color=black] (0,0.25) node {$\mathrm{id}$};
\draw (-1.1,0.25) node {$H$};
\draw (0.2,-0.8) node {$s\cdot$};
\draw (0,-3.35+1.73-0.4) node {$s$};
\draw (-0.75,-1.6) node {$sH$};
\draw (1.7,-2.9) node {$sa_-$};
\draw (-1.5,-2.9) node {$sa_+$};
\draw (-1.61,-2.1) node {$sVt$};
\draw (1.65,-2.1) node {$sVt$};
\draw (2.7,0.5) node {$V$};
\draw (-2.65,0.5) node {$V$};
\end{tikzpicture}
  \caption{The subsets $V$ and $sVt$ are disjoint.}
  \label{fighyp:right}
\end{subfigure}
  \caption{The action of $G$ and a good neighborhood $V$ of $\{a_+,a_-\}$.}
\label{fig1}
\end{figure*}
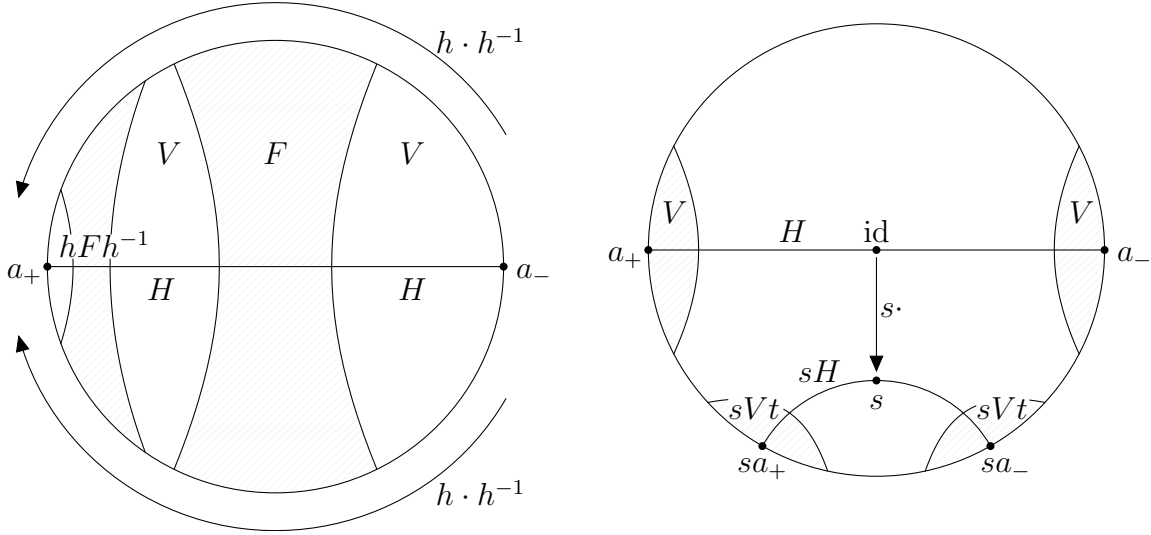

The following fact is certainly known, but we include a proof for the sake of completeness.
Let us recall that the sets $M(a_\pm,A)$ are neighbourhoods of $a_\pm$. 

\begin{lem}
\label{lefthyp}
For any finite sets $A,B \subset V(\Gamma)$,  there exists $n \in \Z$ such that 
\[G \cap (a^n \cdot M(a_-,B)^c) \subset M(a_+,A).\]
\end{lem}
\begin{proof}
First note that we can (and we will) assume that $a_-\notin M(a_+,A)$.
By Lemma \ref{uniformbasis} there exists a finite set $C \subset V(\Gamma)$ such that for all $y \in M(a_+,C)$ we have $M(y,C)\subset M(a_+,A)$. In particular, for all $y \in M(a_+,C)$ and $z \notin M(a_+,A)$ there exists a geodesic between $y$ and $z$ which intersects $C$. 

Choose $n \in \Z$ such that $a^n B \subset M(a_+,C)$ and such that the distance between points of $C$ and $a^n B$ is larger than the diameter $D$ of $V_{0,a_-}(C)$. We claim that this $n$ satisfies the conclusion of the lemma.

Assume by contradiction that there exists $z \in G = V(\Gamma)$ such that $z \notin a^n M(a_-,B)$ and $z \notin M(a_+,A)$.
Since $z \notin a^n M(a_-,B) = M(a_-,a^n B)$, there exists a geodesic $\alpha \in \cf(a_-,z)$ which contains a point $y \in a^n B \subset M(a_+,C)$. Let us denote $\alpha_{a_-}$ the sub-geodesic of $\alpha$ between $a_-$ and $y$ and with $\alpha_z$ the sub-geodesic between $y$ and $z$. 

Since $a_- \notin M(a_+,A)$, there exists a geodesic between $a_-$ and $y$ which intersects $C$. By Corollary \ref{cor:bigA}, the geodesic $\alpha_{a_-}$ meets $V_{0,a_-}(C)$ at a vertex $x_1$.
Moreover $z \notin M(a_+,A)$, so replacing $\alpha_z$ by another geodesic between $y$ and $z$ if necessary, we can assume that $\alpha_z$ meets $C \subset V_{0,a_-}(C)$ at a vertex $x_2$ (while $\alpha = \alpha_{a_-} \cup \alpha_z$ is still a geodesic).
But then
\[d(x_1,x_2) \leq \diam(V_{0,a_-}(C)) = D.\]
On the other hand, the length of $\alpha$ between these two points is equal to $d(x_1,y) + d(y,x_2)$, while $d(x_1,y) > D$ because $x_1 \in C$ and $y \in a^n B$. This is absurd.
\end{proof}

\begin{lem}
\label{righthyp}
For any $A \subset V(\Gamma)$ finite, there exists a finite $B \subset V(\Gamma)$ such that for any $k \in \Z$, \[(M(a_+,B) \cap (G \setminus H)) a^k \subset M(a_+,A).\]
\end{lem}
\begin{proof}

We want to construct a neighbourhood $M(a_+,B)$ such that for any $y\in M(a_+,B)$, the sequence $(ya^k)_k$ stays inside $M(a_+,A)$. It is helpful to think about $(ya^k)_k$ as a quasi-geodesic, from $ya_-$ to $ya_+$. 

The proof goes in two steps. We firstly find an intermediate neighbourhood $M(a_+,B')$ such that if the sequence $(ya^k)_k$ leaves $M(a_+,A)$ for some $y\in M(a_+,B')$, then it has to go through a fixed finite set, which we can assume to be $B'$. We will conclude by choosing $M(a_+,B)\subset M(a_+,B')$ which does not intersect the sequences $(ya^k)_k$ for $y\in B'$. 

{\sc Step 1.} There exists a finite set $B'\subset V(\Gamma)$ such that if $y\in M(a_+,B') \cap G$ is such that $ya^k\notin M(a_+,A)$ for some $k \in \Z$, then there exists $m\in \Z$ such that $ya^m\in B'$.

By \cite[Proposition 8.21]{GdH90}, there exists a finite constant $r > 0$ such that for any $p \in \Z$, all geodesics between the neutral element $e$ and $a^p$ are contained in the $r$-neighbourhood of the sequence $\{a^k, k \in \Z\}$.

By Lemma \ref{uniformbasis} there exists a finite set $C \subset V(\Gamma)$ such that for all $y \in M(a_+,C)$ we have $M(y,C) \subset M(a_+,A)$.
Put $B':=B(C,r)$, the $r$-neighbourhood of $C$. 

Take $y\in M(a_+,B')\cap G \subset M(a_+,C)$ such that  $ya^k \notin M(a_+,A)$ for some $k \in \Z$ . Then $ya^k\notin M(y,C)$, so there exists a geodesic $\alpha$ between $y$ and $ya^k$ which meets $C$ at a point $c$. Then $y^{-1}c$ belongs to a geodesic between $e$ and $a^k$, so it is at distance less than $r$ to some $a^m$. In other words, $ya^m \in B(C,r) = B'$, which proves Step 1.
  
{\sc Step 2.} Choice of $B$.

Observe that the set of cluster points of the sequences $(ya^k)_k$, with $y \in B'\setminus H$ is finite and contained in $\partial\Gamma\setminus\{a_+,a_-\}$. So there exists $B$ such that
 \[M(a_+,B)\subset M(a_+,B')\quad \text{ and }\quad M(a_+,B)\cap
 \lbrace ba^k\, \vert \, b\in B'\setminus H,k\in \Z\rbrace=\emptyset.\]
The subset $B$ satisfies the conclusion of the lemma. Indeed, if $y
\in M(a_+,B)\cap (G \setminus H)$ is such that $ya^k\notin M(a_+,A)$
for some $k \in \Z$, then by the claim there exists $h$ such that $y
\in B' a^{-h}$. But in this case we would have $y\in M(a_+,B)\cap \lbrace
ba^p\, \vert \, b\in B'\setminus H,p\in \Z\rbrace$, which was assumed to be empty. Therefore $ya^k \in M(a_+,A)$ for any $k$.
\end{proof}

Now we can deduce a relevant property of the conjugacy action of $H$, as shown in Figure \ref{fighyp:left}.

\begin{prop}
\label{conjughyp}
  For every $s,t\in G\setminus H$, there exists an $H$-roaming set $F\subset G\setminus H$ such that $sF^ct\cap F^c$ is finite. 
\end{prop}
\begin{proof}
Choose a neighbourhood $V_0$ of $\{a_+,a_-\}$ such that $V_0$ is disjoint from $sV_0$. Since the right action of $t$ on $\Delta\Gamma$ is continuous, we can find a $V\subset V_0$ such that $V$ and $sVt$ are disjoint (see Figure \ref{fighyp:right}). We observe that $sVt \cap H$, $sHt \cap V$ and $sHt \cap H$ are finite because the only cluster points of $H$ are in $V$ and the only cluster points of $sHt$ are in $sVt$.

Therefore the set $F :=V^c\cap(G\setminus H)$ is such that $sF^ct \cap F^c$ is finite. To prove that it is $H$-roaming, let us construct a disjoining sequence $(h_k)_k$ inductively. First put $h_0 := e$.

Now assume that $h_0,\dots,h_{n-1}$ have been constructed, for some $n \geq 1$. We will construct $h_n$. Denote by $V_n := \bigcap_{i = 0}^n h_iVh_i^{-1}$. It is a neighbourhood of $\{a_-,a_+\}$, by continuity of left and right actions of $H$. Now put $F_n := V_n^c\cap(G\setminus H)$.

Recall that the family $\{M(a_\pm,A)\}_A$ forms a basis of neighbourhoods of $a_\pm$. By Lemma \ref{righthyp}, there exists a neighbourhood $V'$ of $a_+$ such that $(V'\cap (G\setminus H)) a^k \subset V_n$ for all $k \in \Z$. By Lemma \ref{lefthyp}, there exists $k_n \in \Z$ such that $G\cap a^{k_n} V^c \subset V'$ and in particular  $a^{k_n} F \subset V'$. Note also that $(a^{k_n} F) \cap H = \emptyset$.

Altogether, we get that $a^{k_n} F a^{-k_n} \subset V_n$ is disjoint from $F_n$. But $F_n$ contains all the $h_iFh_i^{-1}$, $i \leq n-1$. So we can define $h_n = a^{k_n}$.
\end{proof}

Now Theorem A follows from Proposition \ref{prop:condition}.

\begin{rem}\label{unique}
Note that in the proof of Proposition \ref{conjughyp} the disjoining sequence that we construct is contained in the subgroup $H_0 := \langle a \rangle \subset H$. Then the proof of Theorem A actually shows that if $P \subset LG$ is an algebra with property Gamma such that $LH_0 \subset P$, then $P \subset LH$. Hence $u_a$ is contained in a unique maximal amenable von Neumann subalgebra of $M$.
\end{rem}


\section{Relatively hyperbolic case}
\label{proofs}

\subsection{Proof of Theorem B}
\label{proofB}

Let $G$ be a hyperbolic group relative to a family $\rg$ of subgroups of $G$, and let $H \in \rg$ be an infinite subgroup.

Consider a Cayley graph $\Gamma$ of $G$ such that  the coned-off graph
$\hat\Gamma$ of $\Gamma$ with respect to $\rg$ is fine and
hyperbolic. Denote by $\Delta\Gamma$ its Gromov compactification,
endowed with the topology generated by the sets $\lbrace M(x,A)\rbrace_{x,A}$. We still identify $G$ with $V(\Gamma)\subset V(\hat \Gamma)$.

Now denote by $c = [H] \in V(\hat{\Gamma})$ the vertex associated with $[H] \in G/H$. This point is not in the boundary $\partial \Gamma$, but it is represented out of $\Gamma$, as in Figure \ref{fig1}.

We will show that for any neighbourhood $V$ of $c$, the set $F :=V^c\cap (G\setminus H)$ (Figure \ref{fig:left}) is $H$-roaming in the sense of Definition \ref{Hroaming}. Then we will show that if $V$ is small enough (Figure \ref{fig:right}), $F$ satisfies the condition of Proposition \ref{prop:condition}, hence proving Theorem B.

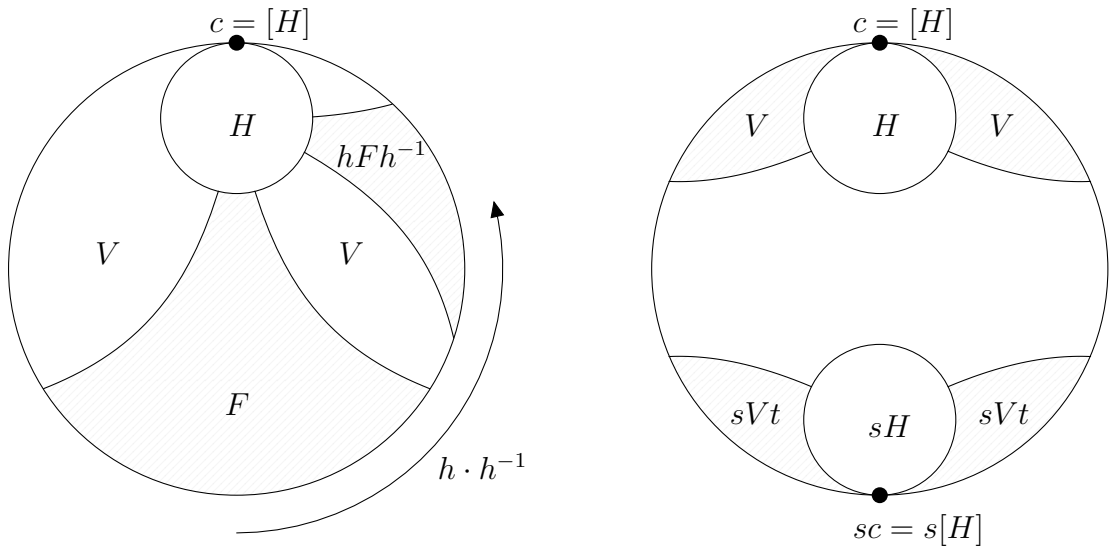
\begin{figure*}[h]
\centering
\begin{subfigure}{.5\textwidth}
  \centering
  \begin{tikzpicture}[line cap=round,line join=round,>=triangle 45,x=1.0cm,y=1.0cm]
\clip(-3,-3.8) rectangle (14,3.5);

\draw(0,0) circle (3cm);
\draw(0,2) circle (1cm);
\draw[black,domain=0.24:2.54] plot({\x},{(2-3*\x)*(1+\x)^(-1)});
\draw[black,domain=-2.54:-0.24] plot({\x},{(2+3*\x)*(1-\x)^(-1)});
\draw[black,domain=0.24:0.64,rotate around={77:(0,2)}] plot({\x},{(2-3*\x)*(1+\x)^(-1)});
\draw[black,domain=-2.20:-0.24, rotate around={77:(0,2)}] plot({\x},{(2+3*\x)*(1-\x)^(-1)});

\fill[pattern = north east lines, opacity=0.2]
plot[domain=256.06:283.94]
({cos(\x)},{2+sin(\x)}) -- plot[domain=0.24:2.54]
({\x},{(2-3*\x)*(1+\x)^(-1)}) -- plot[domain=328:212]
({3*cos(\x)},{3*sin(\x)}) -- plot[domain=-2.54:-0.24] ({\x},{(2+3*\x)*(1-\x)^(-1)}) ;

\fill[pattern = north east lines, opacity=0.2]
plot[domain=0.24:0.64,rotate around={77:(0,2)}]
({\x},{(2-3*\x)*(1+\x)^(-1)}) -- plot[domain=46.8:-17.87]
({3*cos(\x)},{3*sin(\x)}) -- plot[domain=-2.20:-0.24, rotate
around={77:(0,2)}] ({\x},{(2+3*\x)*(1-\x)^(-1)}) -- plot[domain=332.13-360:0]
({cos(\x)},{2+sin(\x)});

\draw[black,domain=-90:15,->] plot({3.5*cos(\x)},{3.5*sin(\x)});

\draw (-0.5,3.6) node[anchor=north west] {$c=[H]$};
\fill [color=black] (0,3) circle (3pt);
\draw (-0.25,2.2) node[anchor=north west] {$H$};
\draw (-0.3,-1.5) node[anchor=north west] {$F$};
\draw[rotate around={77:(0,2)}] (0.1,0.83) node[anchor=north west] {$hFh^{-1}$};
\draw (2.5,-2.3) node[anchor=north west] {$h\cdot h^{-1}$};
\draw (-2,0.5) node[anchor=north west] {$V$};
\draw (1.2,0.5) node[anchor=north west] {$V$};
\end{tikzpicture}
   \caption{Conjugacy action of $H$ on $G$.}
  \label{fig:left}
\end{subfigure}%
\begin{subfigure}{.5\textwidth}
  \centering
  \begin{tikzpicture}[line cap=round,line join=round,>=triangle 45,x=1.0cm,y=1.0cm]
\clip(-3,-3.8) rectangle (3.2,3.5);

\draw(0,0) circle (3cm);
\draw(0,2) circle (1cm);
\draw(0,-2) circle(1cm);

\draw[black,domain=0.24:1.14,rotate around={50:(0,2)}] plot({\x},{(2-3*\x)*(1+\x)^(-1)});
\draw[black,domain=-0.24:-1.14,rotate around={-50:(0,2)}]
plot({\x},{(2+3*\x)*(1-\x)^(-1)});

\draw[black,domain=0.24:1.14,rotate around={-50:(0,-2)}] plot({\x},{(-2+3*\x)*(1+\x)^(-1)});
\draw[black,domain=-0.24:-1.14,rotate around={50:(0,-2)}]
plot({\x},{(-2-3*\x)*(1-\x)^(-1)});

\fill[pattern = north east lines, opacity=0.2] plot[domain=22.77:180-22.77]
({3*cos(\x)},{3*sin(\x)}) -- plot[domain=-0.24:-1.14,rotate around={-50:(0,2)}]
({\x},{(2+3*\x)*(1-\x)^(-1)}) -- plot[domain=180+26.06:90]
({cos(\x)},{2+sin(\x)}) -- 
plot[domain=90:-26.06] ({cos(\x)},{2+sin(\x)})
plot[domain=0.24:1.14,rotate around={50:(0,2)}]
({\x},{(2-3*\x)*(1+\x)^(-1)});

\fill[pattern = north east lines, opacity=0.2] plot[domain=270+67.22:270-67.22]
({3*cos(\x)},{3*sin(\x)}) -- plot[domain=0.24:1.14,rotate
around={-50:(0,-2)}] ({\x},{(-2+3*\x)*(1+\x)^(-1)}) --
plot[domain=270-116.04:270+116.04] ({cos(\x)},{-2+sin(\x)}) --
plot[domain=-1.14:-0.24,rotate around={50:(0,-2)}] ({\x},{(-2-3*\x)*(1-\x)^(-1)});

\draw (1.3,2.2) node[anchor=north west] {$V$};
\draw (-1.3,2.2) node[anchor=north east] {$V$};
\draw (1.15,-2.2) node[anchor=south west] {$sVt$};
\draw (-1.15,-2.2) node[anchor=south east] {$sVt$};
\draw (-0.5,3.6) node[anchor=north west] {$c=[H]$};
\fill [color=black] (0,3) circle (3pt);
\draw (-0.25,2.2) node[anchor=north west] {$H$};
\draw (-0.3,-1.8) node[anchor=north west] {$sH$};
\draw (-0.5,-3.1) node[anchor=north west] {$sc=s[H]$};
\fill [color=black] (0,-3) circle (3pt);
\end{tikzpicture}
  \caption{The subsets $V$ and $sVt$ are disjoint.}
  \label{fig:right}
\end{subfigure}
  \caption{The action of $G$ and a good neighborhood $V$ of $c=[H]$.}
\label{fig2}
\end{figure*}

In this section, we will write $V_r$ instead of $V_{r,c}$, $r \geq 0$ (see Lemma \ref{lem:key}).

Remark that since $c$ shares an edge with all the points in $H$ (and
only with them), any geodesic between $c$ and a point $x \in \Delta
\Gamma$ contains exactly one element in $H$. In particular one has the
following simple lemma. 

\begin{lem}
\label{Hneighbourhood}
The family $\{M(c,A)\}_{A\subset H}$ is a basis of neighbourhoods of $c$.
\end{lem}
\begin{proof}
Let $B\subset V(\hat\Gamma)$ be a finite subset, for every $b\in B$ choose a geodesic $\alpha_b$ from $c$ to $b$. Set $A:=\{\alpha_b(1)\}_{b\in B}$ and observe that $M(c,A)\subset M(c,B)$. 
\end{proof}

\begin{rem}\label{inH}
In the same way, if $A \subset H$ is finite and $r \geq 0$, the set $V_r(A)$ from Corollary \ref{cor:bigA} can be assumed to be contained in $H$. Indeed one can replace $V_r(A)$ by the finite set of points in $H$ which lie on an $r$ quasi-geodesic from $V_r(A)$ to $c$.
\end{rem}

To give an idea about the topology near $c$, let us mention that any sequence $(h_n)_n$ in $H$ which goes to infinity converges to $c$. 

As in the hyperbolic case, we will study geometrically the conjugacy action of $H$ on $G$. We will treat left and right actions separately. First, the left multiplication of $G$ on itself extends to an isometric action on $\hat\Gamma$, and hence extends to a continuous action on $\Delta\Gamma$. Let us extend also the right action.

\begin{defn}
The {\it right action} of $G$ on $\Delta \Gamma$ is the action whose restriction to $G$ is equal to the right multiplication by $G$ on itself, and which is trivial on $\Delta \Gamma \setminus G$. This action is {\it a priori} not continuous, and it clearly commutes with the left action.
\end{defn}

The following lemma is contained in Proposition 12 of \cite{Oz06}, which actually shows the continuity of the right action on $\Delta\Gamma$.

\begin{lem}
\label{Oz}
The right action of $G$ on $\Delta\Gamma$ is continuous at $c$.
\end{lem}
\begin{proof}
Let $g \in G$, and let $(x_n)$ be a sequence converging to $c$. We
want to prove that $(x_n g)$ converges to $c$. Since the right
action is trivial on $\Delta\Gamma\setminus G$, we can assume that
$x_n\in G$ for all $n$. Fix a finite set $A \subset H$. We will show
that $x_n g \in M(c,A)$ for $n$ large enough.

By Lemma \ref{uniformbasis}, there exists a finite set $C \subset V(\hat{\Gamma}) \setminus \{c\}$ such that for all $y \in M(c,C)$ we have $M(y,C) \subset M(c,A)$. So if $y \in M(c,C)$ and $z \notin M(c,A)$, then
there exists a geodesic between $y$ and $z$ which intersects $C$.

Assume by contradiction that there exist infinitely many indices $n$ for which $x_n g\notin M(c,A)$.
By assumption $x_n \in M(c,C)$ for $n$ large enough, which implies that there exists a geodesic $\alpha_n \in \cf(x_n,x_n g)$ which intersects $C$ for infinitely many $n$'s. Then $x_n^{-1}\alpha_n$ belongs to $\cf(e,g)$ and the set $X := \bigcup_{\alpha \in \cf(e,g)} V(\alpha)$ is finite by Lemma \ref{lem:key}(1). 
Altogether we get that $x_n^{-1}C \cap X \neq \emptyset$ for infinitely many $n$'s. Taking a subsequence if necessary, we find an element $c' \in C$ and $x \in X$ such that $x_n^{-1}c' = x$ for all $n$.

This implies that $x_p^{-1}x_n \in \Stab_G(x)$ for all $p,n \in \N$. Since there are infinitely many distinct elements $x_n$, we get that $x$ has to be a conic point, and for all fixed $p$, the sequence $(x_p^{-1}x_n)_n$ converges to $x$. But by continuity of the left action, the sequence also converges to $x_p^{-1}c$.

Therefore $c = x_p x = c'$. This contradicts our assumption that $c \notin C$.
\end{proof}

We now collect properties of left and right actions of $H$ on $\Delta\Gamma$. Note that the left action of $H$ stabilizes $c$ (and $H = \Stab(c)$).

\begin{lem}
\label{leftpara}
For any finite subsets $A,B\subset H$,  there exists $h\in H$ such that \[h M(c,A) ^c \subset M(c,B).\]
\end{lem}
\begin{proof}
  By Remark \ref{inH}, we may assume that $V_0(A) \subset H$. Let $h \in H$ be such that $h
  V_0(A) \cap B = \emptyset$. Let $x\in M(c,A)^c$ and let $\alpha$ be
  any geodesic between $c$ and $h x$, $\alpha\in\cf(c,hx)$ . By Corollary \ref{cor:bigA}(1), $h^{-1} \alpha \in\cf(c,x)$ contains a point $a \in V_0(A)$. Thus $h a$ is the unique
  point of $H$ which is on $\alpha$. In particular $\alpha$ contains
  no point of $B$.
\end{proof}

\begin{lem}
\label{rightpara}
For any $A \subset V(\hat\Gamma)$ finite, there exists a finite $B \subset V(\hat\Gamma)$ such that for any $h\in H$, \[(M(c,B) \cap (G \setminus H)) h \subset M(c,A).\]
\end{lem}
\begin{proof}
By Lemma \ref{Hneighbourhood}, we can assume that $A \subset
H$. Consider an element $x \in G \setminus H$ such that $x \notin
M(c,A)$ and take $h \in H$. We will show that $xh \notin M(c,V_{2}(A))$.

Let $\alpha$ be a geodesic from $c$ to $x$ that meets $A$ and put $a := \alpha(1) \in \alpha \cap A$. Note that since $xh \notin H$, we have $d(xh,c) \geq 2$, and at the same time $d(xh,x) \leq 2$, because $xh$ and $x$ lie in the same coset $xH$. Hence one can choose a projection $z_0$ of $xh$ on $\alpha$ to be different from $c$. Thus the path from $xh$ to $c$ through $z_0$ constructed as in Definition \ref{perpendicular} is a $2$-quasi-geodesic and it contains $a = \alpha(1) \in A$. By Corollary \ref{cor:bigA}(2), this implies that $xh \notin M(c,V_{2}(A))$. Thus $B := V_{2}(A)$ satisfies the conclusion of the lemma.
\end{proof}

As in the hyperbolic case, we deduce the following property of the conjugacy action of $H$ on $G$, see Figure \ref{fig:left}.

\begin{prop}
\label{conjugpara}
  For every $s,t\in G\setminus H$, there exists an $H$-roaming set $F\subset G\setminus H$ such that $sF^ct\cap F^c$ is finite. 
\end{prop}
\begin{proof}
  We proceed as in Proposition \ref{conjughyp}. By continuity of left and right action at $c$ (Lemma \ref{Oz}), there exists a neighbourhood $V$ of $c$ such that $V$ and $sVt$ are disjoint (see Figure \ref{fig:right}). We observe that $sVt \cap H$, $sHt \cap V$ and $sHt \cap H$ are finite because the cluster point of $H$ lies in $V$ and the cluster point of $sHt$ lies in $sVt$.

Therefore, the set $F :=V^c\cap(G\setminus H)$ is such that $sF^ct \cap F^c$ is finite. One can deduce that this set is $H$-roaming from Lemma \ref{rightpara} and Lemma \ref{leftpara}, as in Proposition \ref{conjughyp}.
\end{proof}

Now Theorem B follows from Proposition \ref{prop:condition}.

\begin{rem}
\label{remB}
For later use, note that the set $F$ in Proposition \ref{conjugpara} can be chosen such that $sF^ct\cap F^c\subset H$ and hence $s(F \cup H)^ct \subset F \cup H$.
\end{rem}

\subsection{Proof of Corollary C}

Assume that $G$ is hyperbolic relative to a family $\rg$ of amenable subgroups, and consider an infinite maximal amenable subgroup $H < G$.
We will show that $G$ is hyperbolic relative to $\rg \cup \{ H \}$. Then Theorem B will directly allow to conclude that $LH$ is maximal amenable inside $LG$.

The argument relies on Osin's work \cite{Os06a,Os06b}.

\begin{defn}
An element $g \in G$ is said to be {\it hyperbolic} if it has infinite order and is not contained in a conjugate of a group in $\rg$.
\end{defn}

\begin{defn}
A subgroup $K$ of $G$ is said to be {\it elementary} if it is either finite, or contained in a conjugate of a group in $\rg$, or if it contains a finite index cyclic subgroup $\langle g \rangle$, for some hyperbolic element $g$.
\end{defn}

The (Gromov-)Tukia's strong Tits alternative (see \cite[Theorem 2T, Theorem 3A]{Tu94} using \cite[Definition 1]{Bo12}) states that a non-elementary subgroup $K$ of $G$ contains a copy of the free group on two generators.

In particular, our amenable subgroup $H$ is elementary. If it is contained in a conjugate $aH_ia^{-1}$ of a group in $\rg$, then it is equal to $aH_ia^{-1}$ by maximal amenability, and Theorem B gives the result. 

Now assume that $H$ contains a finite index cyclic subgroup $\langle g \rangle$, for some hyperbolic element $g$. Osin showed in \cite[Section 3]{Os06b} that such a hyperbolic element $g$ is contained in a unique maximal elementary subgroup $E(g)$ (thus $H = E(g)$, by maximal amenability). Moreover he showed \cite[Corollary 1.7]{Os06b} that $G$ is hyperbolic relative to $\rg \cup \{E(g)\}$. This is what we wanted to show.

\section{Product case: proof of Theorem D}

Observe that if $H_i<G_i$, for $i=1,2$, are infinite maximal amenable subgroups, then the von Neumann subalgebra $L(H_1\times H_2)<L(G_1\times G_2)$ is neither maximal Gamma nor mixing as soon as $H_1\neq G_1$. 

Therefore to treat the product case, we will have to deal with relative notions. We could consider a relative notion of property Gamma and proceed as in Section \ref{central sequences}. We choose instead to apply directly the work of C. Houdayer and the {\it relative asymptotic orthogonality property}, \cite{Ho14b}. Note that in the case of virtually abelian subgroups $H_1$, $H_2$ we could also use \cite[Theorem 2.8]{CFRW08}.

\begin{defn}
\label{wmixingthrough}
Let $A \subset N \subset (M,\tau)$ be finite von Neumann algebras. The inclusion $N \subset M$ is said to be {\it weakly mixing through $A$} if there exists a sequence of unitaries $(v_n)_n \subset \cU(A)$ such that \[\lim_n \Vert E_N(xv_ny)\Vert_2 = 0, \, \forall x,y \in M \ominus N.\]
\end{defn}

\begin{exmp}
\label{exmp}
If $H < G$ is an inclusion of groups satisfying the assumption of Proposition \ref{prop:condition} (\emph{e.g.} if $H$ and $G$ are as in Theorem B), then for any trace-preserving action $G \curvearrowright (Q,\tau)$ on a finite von Neumann algebra, the inclusion $Q \rtimes H \subset Q \rtimes G$ is weakly mixing through $LH$. The proof is the same as the proof of Proposition \ref{propmixing}.
\end{exmp}

\begin{defn}[\cite{Ho14b}, Definition 5.1]
\label{relativeAOP}
Let $A \subset N \subset (M,\tau)$ be an inclusion of finite von Neumann algebras. We say that $N \subset M$ has the {\it asymptotic orthogonality property relative to $A$} if for every $\Vert \cdot \Vert_\infty$-bounded sequences $(x_n)_n$ and $(y_n)_n$ in $M \ominus N$ which asymptotically commute with $A$ in the $\|\cdot\|_2$-norm, we have that
\[\lim_n \langle ax_nb,y_n \rangle = 0, \text{ for all } a,b \in M \ominus N.\]
\end{defn}

\begin{thm}[\cite{Ho14b}, Theorem 8.1]
\label{thmrelativeAOP}
Let $A \subset N \subset (M,\tau)$ be an inclusion of finite von Neumann algebras.
Assume the following:
\begin{enumerate}
\item $A$ is amenable.
\item The inclusion $N \subset M$ is weakly mixing through $A$.
\item The inclusion $N \subset M$ has the relative asymptotic orthogonality property relative to $A$.
\end{enumerate}
Then any amenable von Neumann subalgebra of $M$ containing $A$ is automatically contained in $N$.
\end{thm}

From now on, we consider the crossed-product von Neumann algebras $Q\rtimes G$ associated to a trace-preserving actions $G \curvearrowright (Q,\tau)$. As for group von Neumann algebras, denote by $u_g$ the unitaries of $Q \rtimes G$ corresponding to elements $g \in G$ and for any set $F \subset G$ denote by $P_F : L^2(Q,\tau) \otimes \ell^2(G) \to L^2(Q,\tau) \otimes \ell^2(F)$ the orthogonal projection.

\begin{prop}
\label{relativecondition}
Let $H < G$ be an inclusion of two infinite groups, with $H$ amenable. Consider an action $G \curvearrowright (Q,\tau)$ of $G$ on a tracial von Neumann algebra, and assume that for any $s,t \in G \setminus H$, there exists an $H$-roaming set $F \subset G\setminus H$ such that $s(F \cup H)^ct\subset F \cup H$.

Then the inclusion $Q \rtimes H \subset Q \rtimes G$ has the asymptotic orthogonality property relative to $LH$.
\end{prop}
\begin{proof}
Consider two $\Vert \cdot \Vert_\infty$-bounded sequences $(x_n)_n$ and $(y_n)_n$ in $(Q\rtimes G) \ominus (Q\rtimes H)$ which asymptotically commute with $LH$. By linearity and density it is sufficient to check that for any $s,t \notin H$, \[\lim_n \langle u_sx_nu_t,y_n \rangle = 0.\]
Fix $s,t \in G \setminus H$. There exists an $H$-roaming set $F$ such that $s(F \cup H)^ct\subset F \cup H$. Proceeding as in the proof of Lemma \ref{lemmalnormal}, it is easy to show that $\lim_n \Vert P_F(x_n)\Vert_2 = \lim_n \Vert P_F(y_n) \Vert_2 = 0$.
Note also that for all $n$, we have $x_n = P_{H^c}(x_n)$ and $y_n = P_{H^c}(y_n)$. Therefore
\begin{align*}
\lim_n \langle u_sx_nu_t,y_n \rangle & = \lim_n \langle u_sP_{F^c}(x_n)u_t,P_{F^c}(y_n) \rangle\\
& = \lim_n \langle u_sP_{(F \cup H)^c}(x_n)u_t,P_{(F \cup H)^c}(y_n) \rangle = 0,
\end{align*}
because $s(F \cup H)^ct\subset F \cup H$. This ends the proof of the proposition.
\end{proof}

\begin{proof}[Proof of Theorem D]
For $i=1,\dots,n$, let $G_i$ be a hyperbolic group relative to a family $\rg_i$ of subgroups and let $H_i\in \rg_i$ be an infinite amenable group. Consider the inclusion 
\[H:=H_1 \times \cdots \times H_n<G:= G_1\times \cdots \times G_n.\]
Let $(Q,\tau)$ be a finite amenable von Neumann algebra and consider a trace-preserving action $G \curvearrowright (Q,\tau)$ of $G$. Put $N := Q \rtimes H$ and $M := Q \rtimes G$.

Assume that $P$ is an intermediate amenable von Neumann subalgebra: $N \subset P \subset M$. We have to show that $P = N$. In order to do so, we will show that for all $i = 1,\dots, n$, we have 
\[ P \subset N_i:= Q \rtimes (G_1 \times \cdots \times G_{i-1} \times H_i \times G_{i+1} \times \cdots \times G_n).\]
This is enough to conclude, because $N = \cap_{i=1}^n N_i$.

For $i \in \{1,\dots,n\}$, we set $A_i := LH_i$ and $Q_i := Q \rtimes \hat G_i$, where $\hat G_i$ is the direct product of all $G_j$, $j \neq i$. Then we have $N_i \simeq Q_i \rtimes H_i$ and $M \simeq Q_i \rtimes G_i$.

By Proposition \ref{conjugpara} (and Remark \ref{remB}), we see that $H_i \subset G_i$ satisfies the assumptions of Proposition \ref{relativecondition} so that $N_i \subset M$ has the asymptotic orthogonality property relative to $A_i$.
Moreover the Example \ref{exmp} tells us that $N_i \subset M$ is (weakly) mixing through $A_i$.

By Theorem \ref{thmrelativeAOP}, one concludes that the amenable algebra $P$, which contains $A_i$, is contained in $N_i$. This ends the proof of Theorem D.
\end{proof}


\bibliographystyle{alpha1}

\end{document}